\documentclass{article}

\usepackage{arxiv}

\usepackage[utf8]{inputenc} % allow utf-8 input
\usepackage[T1]{fontenc}    % use 8-bit T1 fonts
\usepackage{url}            % simple URL typesetting
\usepackage{booktabs}       % professional-quality tables
\usepackage{nicefrac}       % compact symbols for 1/2, etc.
\usepackage{microtype}      % microtypography
\usepackage{lipsum}
\usepackage{amsmath, amsthm, amscd, amsfonts, amssymb, graphicx, color}
\usepackage[backref,colorlinks=true]{hyperref}

\title{On Some Properties of Finsler Manifolds of Stretch Curvature}

\author{
  Pejhman Vatandoost-Miandehi $^1$\\
  Department of Mathematics and Computer Science\\
  Amirkabir University of Technology\\
  Tehran, Iran \\
   \And
 Masoud Nikokar $^2$\\
  Department of Mathematics and Computer Science\\
  Amirkabir University of Technology\\
  Tehran, Iran \\
}
\renewcommand{\d}{\partial}
\DeclareMathOperator{\rdiv}{\mathrm{div}}
\DeclareMathOperator{\trace}{\mathrm{trace}}

\begin{document}
\maketitle
\newtheorem{theorem}{Theorem}[section]
\newtheorem{lemma}[theorem]{Lemma}
\newtheorem{proposition}[theorem]{Proposition}
\newtheorem{corollary}[theorem]{Corollary}
\newtheorem{question}[theorem]{Question}

\theoremstyle{definition}
\newtheorem{definition}[theorem]{Definition}
\newtheorem{algorithm}[theorem]{Algorithm}
\newtheorem{conclusion}[theorem]{Conclusion}
\newtheorem{problem}[theorem]{Problem}
\newtheorem{example}[theorem]{Example}

\theoremstyle{remark}
\newtheorem{remark}[theorem]{Remark}
\numberwithin{equation}{section}
\footnotetext[1]{E-mail:~\texttt{(pejhman.vatandoost@gmail.com)}; \texttt{(pejhman.vatandoost@iran.ir)}}
\footnotetext[2]{E-mail:~\texttt{(m.nikokar@iran.ir)}}
\begin{abstract}
Finsler metrics with relatively non-negative (non-positive, respectively), constant and isotropic stretch curvatures are investigated in this paper. In particular, it is proved that every non-Riemannian  $(\alpha, \beta)$-metric with a nonzero constant flag curvature and a non-zero relatively isotropic stretch curvature over a manifold of dimension $n\geq 3$  is of a characteristic scalar constant over the Finsler geodesics. It is also shown that every compact Finsler manifold with a relatively non-negative (non-positive, respectively) stretch curvature is a Landsberg metric. Finsler manifolds with $2$-dimensional relative stretch curvature are also investigated. 
\end{abstract}

% keywords can be removed

\section{Introduction}
In Finsler geometry, there are many non-Riemannian quantities. These include Cartan torsion $C$, Berwald curvature $B$, the Landsberg curvature $L$, mean Landsberg curvature $J$, stretch curvature $\Sigma$ and so on. Investigating these non-Riemannian quantities, all of which are equal to zero in Riemannian geometry, introduces us to the nature of Finsler geometry.

A Finsler metric $F$ on a smooth manifold $M$ is said to be a Bervald metric if its Berwald curvature is equal to zero. Or, in other words, the $G^i$ coefficients of the spray defined on $M$, i.e.,
\begin{equation*}
G(x,y)=y^i \dfrac{\d}{\d x^i}- 2G^i \dfrac{\d}{\d y^i},
\end{equation*}
are squared. That is, there are scalar functions ${\Gamma^i}_{jk}(x)$ such that  $\dfrac{\d^2 G^i}{\d y^i \d y^j} = {\Gamma^i}_{jk}(x)$. Another family of Finsler metrics that includes the Berwall metric is the Landsberg metric. This family is Finsler metrics with Landsberg curvature tensor equal to zero. The Landsberg  tensor $L_y$ for every $y \in T_x M_0$ is equal to the rate of Cartan tensor changes along geodesics. Assuming $\{e_i\}$ is an orthonormal basis of space $(T_x M, g_y)$, then $J_y:= \sum_{i=1}^{n}L_y(e_i, e_i, \cdot)$ is called the mean Landsberg curvature. Fisnsler metric $F$ is said to be weak Landsberg if $J=0$ \cite{10}.

Browald introduced the notion of the stretch Curvature as a generalization of the Landsberg curvature \cite{3}. He revealed that the stretch curvature $\Sigma$ is equal to zero if and only if the length of a vector under parallel transmission on an infinitesimal parallelogram remains constant. Matsomoto then introduced this curvature as $\Sigma_{ijkl}:= 2 (L_{ijk|l}- L_{ijl|k})$ \cite{6}. Clearly, every Landsberg metric has a zero-stretch curvature. Also, every metric with a zero stretch curvature is a metric with a relatively isotropic non-positive or non-negative stretch curvature, but the opposite is not necessarily true. Therefore, it is important to consider the circumstances in which the opposite is true. In this paper, we are going to study the Finsler metrics with relative stretch curvature (relatively non-negative or non-positive, or constant) and prove the following theorems.

\begin{theorem}\label{thm 1}
	a compact Finsler manifold with a relatively non-negative (non-positive, respectively) stretch curvature is Landsbergian. Also, a complete Finsler manifold with a relatively constant stretch curve and a bounded Landberg curvature is Landsbergian. 
\end{theorem}

\begin{theorem}\label{thm 2}
	Each non-Riemannian $(\alpha, \beta)$-metric with the non-zero constant flag curvature and the non-zero relative stretch curvature over a manifold of dimension $n\geq 3$ is of scalar constant characteristic on the Finsler geodesics.
\end{theorem}

\begin{theorem}\label{thm 3}
	Suppose $F$ is a $2$-dimensional metric with relatively isotropic stretch curvature $c$. Then $F$ is Riemannian if and only if its principal scalar satisfies
	\begin{equation*}
	2 \mu' + 2 \mu^2 F - c\mu F \neq 0,
	\end{equation*}
	in which, $I$ is the principal scalar of $F$, $\mu:=I^{-1} I_{|1}$, and $\mu'=\mu_{|s}y^s$ is the covariant derivative $\mu$ along an arbitrary geodesic.
\end{theorem}

Suppose $M$ is a $n$-dimensional manifold $C^\infty$, $T_x M$ represents a tangent space to $M$ at point $x \in M$, $TM=\cup_{x \in M} T_x M$ is a tangent bundle over the manifold $M$, and $TM_0 = TM\setminus\{0\}$ is a punctured tangent bundles. A Finsler metric over $M$ is a function $F:TM\rightarrow [0, +\infty)$ that has the following properties
\begin{itemize}
	\item[1.] $F$ is a $C^\infty$   mapping over $TM_0$.
	\item[2.] $F$ has positive homogeneity of degree $1$ over the tangent bundle fibers. 
	\item[3.]  For every  $y \in T_x M$, the form of $g_y$  defined as following is definite positive
	\begin{equation*}
	g_{y}(u,v) = \dfrac{1}{2} \dfrac{d^2}{dt^2}[F^2(y+su+tv)]_{s,t=0},
	\qquad
	u,v\in T_xM.
	\end{equation*}
\end{itemize}
Let $F_x:= F_{|T_xM}$ for $x \in M$. To measure the non-Euclidean of $F_x$, Cartan tensor $C_y: T_xM \otimes T_xM \otimes T_xM \to R$ is defined as follows
\begin{equation*}
C_y(u,v,w)=\dfrac{1}{2}\dfrac{d}{dt}[g_{y+tw}(u,v)]_{t=0},
\qquad
u,v,w\in T_xM.
\end{equation*}
The family $C:= \{C_y\}_{y\in TM_0}$ is called Carton torsion. $F$ is a Riemannian metric if and only if $C = 0$. For the vector $y \in T_x M$, the mean Cartan torsion $I_y :T_xM \to R$ is defined as $I_y=I_i(y)dx^i$ where $I_i:= g^{ij} C_{ijk}$ and $g^{ij}=(g_{ij})^{-1}$.

Let $(M, F)$ be a Finsler manifold of dimension $n \geq 3$. Then $F$ is an analytic semi $C$-metric whose Cartan tensor is as follows:
\begin{equation*}
C_{ljm}=\dfrac{p}{(n+1)}(I_l h_{jm} + I_j h_{lm} + I_m h_{jl})
+ \dfrac{q}{\| I\|^2} I_l I_j I_m,
\end{equation*}
where $p:=p(x, y)$ and $q=q(x,y)$ are scalar functions over TM with condition $p+q=1$,  $\|I\|^2 := I_k I^k$ and $h_{ij}:= g_{ij}-F^{-2} y_iy_j$ is angular metric \cite{8}.

The quantity $p$ is called the characteristic scalar of the metric $M$. If $p = 1$, then $F$ is called $C$-analytic.

\begin{theorem}[\cite{8}]\label{thm 4}
	Every non-Riemannian   $(\alpha, \beta)$-metric over a manifold of dimension $n \geq 3$  is a $C$-semi-analytic metric. 
\end{theorem}

The first person to introduce the concept of connection for Finsler metrics was Berwald \cite{3}. After his valuable work, several connections were introduced in various ways, the most famous being Chern, Cartan and Berwald connection. See \cite{10} for further discussion.

The following statements are used in the computational process of this article. 
\begin{theorem}[\cite{5}]\label{thm 5}
	Suppose $M$ is an oriented manifold with the volume form $\omega$ and $\nabla$  is a torsion-free connection where  $\nabla_\omega=0$. Then, for every vector field $X$ over $M$, $y \in T_x M$ with  $x \in M$, we have
	\begin{equation*}
	(\rdiv X)_x = - \trace(Y \to \nabla_YX) = \nabla_i X^i.
	\end{equation*}
\end{theorem}

\begin{theorem}[\cite{5}]\label{thm 6}
	Suppose $M$ is a compact, oriented manifold with a volume element $\omega$. Then, for every vector field $X$ over $M$, we have
	\begin{equation*}
	\int_M (\rdiv X) \omega =0.
	\end{equation*}
\end{theorem}
This Theorem is not true for non-compact manifolds, but for the vector field $X$ with compact support is hold.

Suppose $(M, F)$ is a $n$-dimensional Finsler curvature. Also suppose that $\nabla$  is a Berwald connection and $\{e_i\}^n_{i=1}$ is an orthonormal basis field (relative to g) for the return bundle $\pi^\star TM$, such that $e_n=l$, and $l$ is the focal cut $l = y / F$. Consider $\{\omega^i\}^n_{i=1}$ as dual basis fields and put
\begin{equation*}
\nabla e_i = \omega_i^j \otimes e_j,
\end{equation*}
where, $\{\omega_i^j\}$ are the forms of connection $\nabla$  with respect to $\{e_i\}^n_{i=1}$ . It is easy to see that $\{\omega^i, \omega^{n+i}\}$ is a local basis for $T^\star(TM_0)$ where,
\begin{equation*}
\omega^{n+i} = \omega_n^i + d(\log F) \delta_n^i.
\end{equation*}

$2$.form  $\{\Omega_i^j\}$ over $TM_0$ is stated as follows
\begin{equation*}
\Omega_i^j = \dfrac{1}{2} 
R_{jkl}^i \omega^k \wedge \omega^l
+ B_{jkl}^i \omega^k \wedge \omega^{n+l}.
\end{equation*}

Let $\{ \bar{e}_i, \dot{e}_i \}_{i=1}^n$ be a local basis for   and the dual basis $\{\omega^i, \omega^{n+i}\}$ is related to  $T^\star(TM_0)$. The above mentioned R and B are called $hh$-curvature and $hv$-curvature, respectively \cite{10}.

Using the Berwald connection, one can define the covariant derivative of functions over $TM_0$. For example, if f is a scalar function, then $f_{|i}$ and $f_{.i}$ are defined as follows
\begin{equation*}
df = f_{|i}\omega^i + f_{.i}\omega^{n+i}
\end{equation*}
where $"|"$ and $"."$ are the symbol of $h$-covariant derivative (horizontal derivative) and $v$-covariant derivative (vertical derivative) with respect to the Berwald connection $F$, respectively.
% % % % % % % %Lemma 1
\begin{lemma}[\cite{10}]\label{lem 1}
	The following Bianchi identities are hold for the Berwald connection
	\begin{equation*}
	\aligned
	& {R^i_j}_{kl.m} = {B^i}_{jml|k} - {B^i}_{jmk|l},\\
	& {B^i}_{jkl.m} = {B^i}_{jkm.l}.
	\endaligned
	\end{equation*}
\end{lemma}

The horizontal derivative of Carton torsion along the geodesics defines the Landsberg curvature as follows.
\begin{align*}
&  L_y: T_xM \otimes T_xM \otimes T_xM \to R \\
& L_y (u,v,w)= L_{ijk}(y) u^i  v^j w^k,
\end{align*}
where, $L_{ijk}=C_{ijk|s} y^s$, $u= u^i \dfrac{\d}{\d x^i}|_x$, $v= v^i \dfrac{\d}{\d x^i}|_x$ and also $w= w^i \dfrac{\d }{\d x^i}|_x$. Family $L:=\{ L_y \}_{y\in TM_0}$ called the Landsberg curvature. A Finsler metric is called Landsberg if $L = 0$.

The following equations are easily illustrated by using the properties of $2$-forms of Berwald connection curvature \cite{10}:
\begin{equation*}
g_{ij|k}=-2 L_{ijk},~~~
g_{ij.k}=2 C_{ijk}.
\end{equation*}
Also it can also be concluded
\begin{equation*}
L_{ijk}=-\dfrac{1}{2} y^s g_{sm} {B^m}_{ijk}.
\end{equation*}

After introducing the notion of stretch curvature as a generalization of the Landsberg curvature by Berwald, Matsomoto introduced its form as follows:
\begin{align*}
&  \Sigma: T_xM \otimes T_xM \otimes T_xM \otimes T_xM \to R \\
& \Sigma_y (u,v,w,z) = \Sigma_{ijkl}(y) u^i v^j w^k z^l
\end{align*}
where $\Sigma_{ijkl}:= 2 (L_{ijk|l} - L_{ijl|k})$. A Finsler metric is said to be a stretch metric if  $\Sigma=0$. A Finsler metric $F$ over the manifold $M$ is said to be of a relative stretch curvature (with ratio $c$) if
\begin{equation*}
\Sigma_{ijkl}:= cF (C_{ijk|l} - C_{ijl|k}),
\end{equation*}
where $c=c(x,y)$. The stretch curvature $F$ is said to be relatively nonnegative (non-positive, respectively) if $c=c(x,y)$ is a non-positive (nonnegative, respectively) function. The stretch curvature of the Finsler metric $F$ over the curvature $M$ is said to be relatively isotropic if $c=c(x)$ is a scalar function over $M$. $F$ is of the relatively constant stretch curvature if $c$ is a constant real number.

\begin{example}
	The metric $F$ is called $R$-square if  ${R_j}^i_{kl.m}=0$. If we multiply the second identity of lemma \ref{lem 1} to $y_i$, we get
	\begin{equation*}
	\Sigma_{jmkl} = y_i {R_j}_{kl.m}^i.
	\end{equation*}
	As a result, every Finsler $R$-square metric is a metric with zero stretch curvature.
\end{example}

\begin{example}
	Let
	\begin{equation*}
	F_a(x,y):= \dfrac{\sqrt{|y|^2 - (|x|^2|y|^2 - \langle x,y \rangle^2)  }}{1-|x|^2} 
	+ \dfrac{\langle x,y \rangle}{1-|x|^2}
	+ \dfrac{\langle a,y \rangle}{1-|x|^2},
	\quad
	y\in T_xB^n \sim R^n,
	\end{equation*}
	where $a \in R$  is a fixed vector and  $|a|< 1$. For every  $a \neq 0$, it is easy to see that $F_a$ is locally flat with a negative constant flag curvature. $F$ is a relatively constant stretch metric with $c = -1$.
\end{example}

By taking horizontal covariant derivative from of mean Cartan torsion tensor $I$ along the geodesics, the mean Landsberg tensor  $J_y(u)= J_i(y) u^i$ is obtained, where  $J_i:= I_{i|s} y^s$. The mean Landsberg curvature can be obtained from $J_i := g^{kl} L_{ikl}$, too. A Finsler metric is a weak Landsberg if $J = 0$.

For every Finsler metric $(M, F)$, an overall vector field $\mathbf{G}$ is defined by $F$ over $TM_0$, which can be expressed in the local coordinates $(x^i, y^i)$  for $TM_0$ and is called the spray obtained from the metric $F$:
\begin{equation*}
\mathbf{G}(y) = y^i \dfrac{\d }{\d x^i}- 2 G^i (x,y) \dfrac{\d }{\d y^i},
\end{equation*}
in which, $G^i$ are local functions over $TM_0$, shown as below
\begin{equation*}
G^i(x,y) := \dfrac{1}{4} g^{il} 
\left\{ \dfrac{\d^2 F^2}{\d x^k \d y^l} y^k - \dfrac{\d F^2}{\d x^l} \right\}.
\end{equation*}
For the tangent vector $y \in TM_0$
\begin{align*}
&  B_y: T_xM \otimes T_xM \otimes T_xM \to T_xM \\
& B_y(u,v,w) = {B^i}_{jkl} (y) u^i v^k w^l \dfrac{\d}{\d x^i}|_x,
\end{align*}
and
\begin{align*}
&  E_y: T_xM \otimes T_xM \to R \\
& E_y (u,v) = E_{jk}(y) u^j v^k,
\end{align*}
are called the Berwald curvature and the mean Berwald curvature, respectively, where
\begin{equation*}
{B^i}_{jkl}(y) := \dfrac{\d^3 G^i}{\d y^j \d y^k \d y^l}(y),~~~
E_{jk}(y) := \dfrac{1}{2} {{B_j}^m}_{km}(y)
\end{equation*}
$F$ is called Berwald and weak Berwald is $B=0$ and $E=0$, respectively.

The Riemannian Curvature $R_y= {R^i}_k(y) dx^k \otimes \dfrac{\d }{\d x^i} : T_xM \to T_xM$  is a family of linear mappings on tangent space defined as follows
\begin{equation*}
{R^i}_k = 2\dfrac{\d G^i}{\d x^k} - y^i \dfrac{\d^2 G^i}{\d x^j \d y^k} 
+ 2 G^i \dfrac{\d^2 G^i}{\d y^j \d y^k}  
- \dfrac{\d G^i}{\d x^j} \dfrac{\d G^j}{\d x^k}.
\end{equation*}
Also for the Riemannian curvature, the following relation is hold \cite{10}
\begin{equation}\label{eq1}
{R_j}^{i}_{kl} = \frac{1}{3}(R^i_{k.l} - R^i_{l.k})_{.j}.
\end{equation}
A flag curvature in Finsler geometry is an extension of shear curvature in Riemannian geometry first introduced by Browald \cite{3}.

For a flag $P=span \{y, u \} \subset T_xM$  with a bar $y$, the flag curvature is defined as follows 
\begin{equation*}
K(P,y) := \dfrac{g_y(u, R_y(u))}{g_y(y,y) g_y(u,u) - g_y(u,v)^2}.
\end{equation*}

The Finsler metric $F$ is said to be of the scalar curvature if for every  $y \in T_xM$, the flag curvature $K=K(x,y)$ is a scalar function over tangent bundle cut $TM_0$. If $K$ is constant, then $F$ is called a metric with constant curvature.

\section{Stretch curvature of Finsler metrics}

In this section, we prove the main theorems and some of their corollaries.
\begin{proof}[Proof of Theorem \ref{thm 1}]
	Suppose $p$ is a point over a manifold $M$,  and $y,u,v,w\in T_pM$, and $\sigma: (-\infty, +\infty) \to M$ is the geodesic that passes through the point p with a unit velocity such that
	\begin{equation*}
	\dfrac{d\sigma}{dt}(0)=y.
	\end{equation*}
	$U(t), V(t)$ and $W(t)$ are parallel vector fields along $\sigma$ such that $U(0)=u, V(0)=v$ and $W(0)=w$. Then, we put
	\begin{equation*}
	\begin{split}
	& L(t) = L_{\dot{\sigma}} (U(t), V(t), W(t)),\\
	& L'(t) = L'_{\dot{\sigma}} (U(t), V(t) , W(t)).
	\end{split}
	\end{equation*}
	However, the Finsler manifold 
	$(M,F)$ has a relatively non-negative (non-positive, respectively) stretch curvature or is constant. Given the definition and by multiplying the stretch tensor in $y^l$, it is easy to obtain:
	\begin{equation*}
	L_{ijk|l} y^l = cF L_{ijk},
	\end{equation*}
	where $c:=c(x,y)$  is a nonnegative (non-positive, respectively) homogeneous function over $TM_0$ or $c$ is a constant.
	
	First suppose $c:=c(x,y)$ is a nonnegative (non-positive, respectively) function over $TM_0$. By putting  $\varphi (x,y) := L^{ijk} L_{ijk}$, we have
	\begin{align}
	\dot{\varphi} = \varphi_{|m}y^m  & =
	2 g^{ir} g^{js} g^{kt} L_{rst} L_{ijk|m}y^m \notag\\
	&=
	2 L^{ijk} L_{ijk|m} y^m
	= 2cF\varphi \label{eq2}
	\end{align}
	due to $F$ and  $\varphi$ have positive values, if $c$ is nonnegative (non-positive, respectively), then $\dot{\varphi}$  will be nonnegative (non-positive, respectively).
	
	According to Theorem \ref{thm 5} we have
	\begin{equation*}
	\dot{\varphi}(y) = \varphi_{|m} y^m = \xi (\varphi) =
	\overline{\rdiv}(\varphi\xi).
	\end{equation*}
	Note that $\xi= y^i \dfrac{\delta}{\delta x^i}$  is a geodesic vector field over the unit bundle $SM$ and  $\overline{\rdiv}(\xi)=0$ \cite{13}.
	
	Since $M$ is compact, so  $SM$ is compact, too. Also, the volume form $\omega_{SM}$ over the spherical bundle $SM$ is obtained from volume form $\omega$ over $M$ \cite{1}.
	
	According to Theorem \ref{thm 6} we have
	\begin{equation*}
	\int_{SM} \dot{\varphi} \omega_{SM}=0.
	\end{equation*}
	Since $\dot{\varphi}$  is a homogeneous function and its sign is always nonnegative (non-positive, respectively), then   $\dot{\varphi}=0$. Therefore according to \eqref{eq2}, it results in  $\varphi=0$ or $c=0$. If  $\varphi=0$, then  $L_{ijk}=0$. If $c = 0$, then $\Sigma_{ijkl}=0$  and therefore,  $L'(t) = L_{ijk|l}y^l =0$, that is,  $L(t)=L(0)$. So Cartan torsion is equal to
	\begin{equation*}
	C(t) = t L(0) + C(0),
	\end{equation*}
	that if  $t \to \pm \infty$, then the function will not be bounded, and this contradicts the compression of $M$ (as a result of the bounded Cartan torsion). Therfore, $L(0)=0$ and since compact manifolds are always complete, hence, $L(t)=0$, that is, the metric $F$, is a Landsberg metric. Thus, the first part of the theorem is proved.
	
	Now, if $c$ is a constant function, then again by multiplying the stretch tensor in $y^l$, we get  
	$$L_{ijk|m}y^m = cF L_{ijk}.$$ The general answer to this equation is as follows.
	\begin{equation*}
	L(t) = e^{ct}L(0).
	\end{equation*}
	By tending $t$ to  $+\infty$ or $-\infty$ in the above answer, the non-boundary of the Landsberg curvature is obtained which is inconsistent with the assumption. So
	\begin{equation*}  
	L(t) = L(0)=0.
	\end{equation*}
	Thus, the second part of Theorem \ref{thm 1} also is proved.
\end{proof}

\begin{proof}[Proof of Theorem \ref{thm 2}]
	Suppose $F$ is a stretch metric with a nonzero constant curvature $\lambda$ over a manifold $M$. So we have
	\begin{equation*}
	R_k^i = \lambda \{ F^2 \delta^i_k - y^i y_k \},
	\end{equation*}
	and so
	\begin{equation}\label{eq3}
	{R_j}^i_{kl} = \lambda\{ g_{jl}\delta_k^i - g_{jk} \delta_l^i \}.
	\end{equation}
	Given \eqref{eq3} and the second identity of Lemma \ref{lem 1} we have
	\begin{equation}\label{eq4}
	\Sigma_{jmkl} = y_i {R_j}^i_{kl.m} = 2 \lambda \{ C_{jlm} y_k - C_{jkm} y_l \}.
	\end{equation}
	But $F$ has a relative stretch curvature, that is
	\begin{equation}\label{eq5}
	\Sigma_{jmkl} = 2 (L_{jmk|l} - L_{jml|k}) 
	= cF (C_{jmk |l} - C_{jml|k} ).
	\end{equation}
	Therefore
	\begin{equation*}
	2 \lambda \{ C_{jlm} y_k - C_{jkm} y_l  \} 
	= cF  (C_{jmk |l} - C_{jml|k}).
	\end{equation*}
	By multiplying the expression by $y^l$ we have
	\begin{equation}\label{eq6}
	L_{jmk} + 2 \dfrac{\lambda}{c} F C_{jmk} = 0.
	\end{equation}
	By multiplying \eqref{eq6} by $g^{im}$ we have
	\begin{equation}\label{eq7}
	J_k + 2 \dfrac{\lambda}{c} F I_k = 0.
	\end{equation}
	On the other hand, $F$ is a non-Riemannian   $(\alpha, \beta)$-metric with a dimension of  $n\geq 3$, meaning its Cartan tensor is written as follows
	\begin{equation}\label{eq8}
	C_{ljm} = \dfrac{p}{(n+1)} (I_l h_{jm} + I_j h_{lm} + I_m h_{jl})
	+ \dfrac{q}{\|I\|^2} I_l I_j I_m,
	\end{equation}
	where equation $p+q=1$ is hold. Now, by taking the horizontal derivative $|_s$ from \eqref{eq8} and multiplying it by $y^s$, we have
	\begin{align}\label{eq9}
	\aligned
	C_{jmk|0} = L_{jmk} &=
	\dfrac{p}{n+1} S_{jmk} + \dfrac{p'}{n+1} X_{jmk}
	\\
	&~~~+
	\dfrac{1}{\|I\|^2} \left(q'- \dfrac{2q}{\|I\|^2} I^r J_r\right) I_j I_m I_k
	+ \dfrac{q}{\| I\|^2} T_{jmk},
	\endaligned
	\end{align}
	where
	\begin{align*}
	S_{jmk}  & =  J_j h_{mk} + J_m h_{jk} + J_k h_{jm},\\
	X_{jmk}  & =  I_j h_{mk} + I_m h_{jk} + I_k h_{jm},\\
	T_{jmk}  & =  J_j I_m I_k + J_m I_j I_k + J_k I_m I_j.
	\end{align*}
	By substituting \eqref{eq7} in \eqref{eq9} and also using \eqref{eq6} and \eqref{eq8}, we get
	\begin{equation*}
	\dfrac{p'}{n+1} X_{jmk} + \dfrac{1}{\| I\|^2}q' I_j I_m I_k = 0.
	\end{equation*}
	By multiplying the above expression by $I^jI^m$, and using \eqref{eq7}, we have
	\begin{equation}\label{eq10}
	3 \|I\|^2 p' I_k + (n+1) q' I_k =0.
	\end{equation}
	Given that $F$ is non-Riemannian, we conclude that
	\begin{equation}\label{eq11}
	3 \| I\|^2 p' + (n+1) q' =0.
	\end{equation}
	Since $p+q=1$, then   $p'+q'=0$. By substituting it in \eqref{eq11}, we get $p'=0$, that is, the characteristic scalar of the metric F over Finsler geodesics is always constant.
\end{proof}

\begin{corollary}
	Every Finsler metric $(n\geq 3)$  with a nonzero constant flag curvature $\lambda$ and a nonzero relative stretch curvature (with ratio $c$) is Riemannian if and only if the equation  $2 cc' + c^2 F + 4 \lambda F \neq 0$   is hold where  $c' := c_{|m} y^m$.
\end{corollary}
\begin{proof}
	By multiplying \eqref{eq5} by $y^l$ we have
	\begin{equation}\label{eq12}
	L'_{jmk} = \dfrac{c}{2} F L_{jmk},
	\end{equation}
	where $L'_{jmk} := L_{jmk|s} y^s$. By substituting \eqref{eq6} in \eqref{eq12}, we have
	\begin{equation}\label{eq13}
	L'_{jmk}=-\lambda F^2 C_{jmk}.
	\end{equation}
	Also, by taking derivative of \eqref{eq6}, we have
	\begin{equation}\label{eq14}
	L'_{jmk} 
	=
	\dfrac{2\lambda F}{c} (c' C_{jmk} - L_{jmk})
	=
	\dfrac{2\lambda F}{c} \left( c' + \dfrac{2\lambda F}{c} \right)C_{jmk}.
	\end{equation}
	By comparing \eqref{eq13} and \eqref{eq14}, the proof completes.
\end{proof}

\begin{proof}[Proof of Theorem \ref{thm 3}]
	To prove this theorem, we use the Berwald frame. The Browald Frame is an essential tool introduced by Berwald \cite{4} for studying the $2$-D Finsler manifolds.
	
	For a $2$-dimensional Finsler manifold $(M, F)$, a local field of perpendicular frames $(\ell^i, m^i)$  is called a Berwald frame, where  $\ell^i = 
	y^i/F$ and $m^i$ are unit vectors given that $\ell_i m^i=0$  and  $\ell_i =
	g_{ij} \ell^j$. Considering the Berwald frame we have:
	\begin{equation*}
	C_{ijk}= F^{-1} I m_i m_j m_k,
	\end{equation*}
	\begin{equation}\label{eq15}
	{B^i}_{jkl} =- \dfrac{2 I_{|1}}{I} C_{jkl} \ell^i 
	+ \dfrac{I_2}{3 F} \{ h_{jk} h_l^i  + h_{jl} h_k^i + h_{lk} h_j^i \},
	\end{equation}
	where, $I$ is a homogeneous function of zero degree called the principal scalar of metric $F$. By multiplying \eqref{eq15} by $y_i$ we will have
	\begin{equation}\label{eq16}
	L_{jkl}=\mu F C_{jkl}, 
	\end{equation}
	where, $\mu := \dfrac{I_{|1}}{I}$. Hence, by taking horizontal derivative of \eqref{eq16}, we have
	\begin{equation}\label{eq17}
	L_{jkl|s} = \mu_{|s} F C_{jkl} + \mu F C_{jkl|s}.
	\end{equation}
	According to \eqref{eq17}, the stretch Tensor is expressed as \eqref{eq18}
	\begin{equation}\label{eq18}
	\Sigma_{ijkl} = 2 \mu F(C_{ijk|l}- C_{ijl|k}) 
	+ 2F (\mu_{|l} C_{ijk} - \mu_{|k} C_{ijl}).
	\end{equation}
	According to the assumption, $F$ is of a constant relative stretch curvature 
	\begin{equation}\label{eq19}
	\Sigma_{ijkl} = cF (C_{ijk|l} - C_{ijl|k}),
	\end{equation}
	where $c$ is a scalar function over $TM$. By multiplying \eqref{eq18} and \eqref{eq19} by $y^l$ and comparing the obtained equations, equation \eqref{eq20} is obtained:
	\begin{equation}\label{eq20}
	cF L_{jkl} = 2 \mu F L_{jkl} + 2 F \mu' C_{jkl}.
	\end{equation}
	Given \eqref{eq16}, Equation \eqref{eq20} changes to \eqref{eq21}.
	\begin{equation}\label{eq21}
	(2 \mu' + 2 \mu^2 F - c\mu F) C_{jkl}=0.
	\end{equation}
	From \eqref{eq21}, Theorem is proved.
\end{proof}

\begin{corollary}
	Every $2$-dimensional complete and non-Riemannian stretch metric is a Landsberg metric.
\end{corollary}
\begin{proof}
	Suppose that \eqref{eq21} is reduced to  $\mu' + \mu^2 F =0$. Let $p$ be a point on $M$,   $y\in T_pM $ and $\sigma:(-\infty, +\infty) \to  M$ is a geodesic passing through point $p$ at a unit velocity such that  $\dfrac{d\sigma}{dt}(0)=y$. Over the geodesic $\sigma$, the obtained equation is written as  $\mu' + \mu^2 =0$. The answer to this differential equation is written as \eqref{eq22}
	\begin{equation}\label{eq22}
	\mu(t) = \dfrac{\mu(0)}{t \mu(0)+1}.
	\end{equation}
	If $t \to \pm \infty$, then  $\mu(t)=0$. By substituting it in \eqref{eq16}, we can easily conclude that $F$ is a Landsberg metric.
\end{proof}


\begin{thebibliography}{widest-label}
	
	\bibitem{1}
	H. Akbar-Zadeh, \textit{Les espaces de Finsler et certaines de leurs g\'en\'eralisations}, Ann. Sci. \'Ecole Norm. Sup. (3), \textbf{80} (1963), 1--79.
	
	\bibitem{2} 
	L. Berwald, \textit{\"Uber Parallel\"ubertragung in R\"aumen mit allgemeiner Maßbestimmung}, Jber. Deutsch. Math.-Verein., \textbf{34} (1926), 213--220.
	
	\bibitem{3}
	L. Berwald, \textit{Untersuchung der Kr\"ummung allgemeiner metrischer R\"aume auf Grund des in ihnen herrschenden Parallelismus}, Math. Z., \textbf{25} (1926), 40--73.
	
	\bibitem{4}
	L. Berwald, \textit{On Finsler and Cartan geometries. III. Two-dimensional Finsler spaces with rectilinear extremals}, Ann. of Math. (2), \textbf{42} (1941), 84--112.
	
	\bibitem{5} 
	S. Kobayashi, K. Nomizu, \textit{Foundations of Differential Geometry}, John Wiley \& Sons, New York-London-Sydney, (1969).
	
	\bibitem{6} 
	M. Matsomoto, \textit{An improvement proof of Numata and Shibata’s theorem on Finsler spaces of scalar curvature}, Publ. Math. Debrecen, \textbf{64} (2004), 489--500.
	
	\bibitem{7}
	M. Matsumoto, C. Shibata, \textit{On semi-$C$-reducibility, $T$-tensor$=0$ and $S4$-likeness of Finsler spaces}, J. Math. Kyoto Univ., \textbf{19} (1979), 301--314.
	
	\bibitem{8}
	M. Matsumoto, \textit{On Finsler spaces with Randers metric and special forms of important tensors}, J. Math. Kyoto Univ., \textbf{14} (1974), 477--498.
	
	\bibitem{9}
	B. Najafi, S. Saberali, \textit{On a class of isotropic mean Landsberg metrics}, Differ. Geom. Dyn. Syst., \textbf{18} (2016), 72--80.
	
	\bibitem{10}
	Z. M. Shen, \textit{Differential Geometry of Spray and Finsler Spaces}, Kluwer Academic Publishers, Dordrecht, (2001).
	
	\bibitem{11}
	Z. M. Shen, \textit{On $R$-quadratic Finsler spaces}, Publ. Math. Debrecen, \textbf{58} (2001),  263--274.
	
	\bibitem{12}
	A. Tayebi, H. Sadeghi, \textit{On Cartan torsion of Finsler metrics}, Publ. Math. Debrecen, \textbf{82} (2013), 461--471.
	
	\bibitem{13}
	B. Y. Wu, \textit{A global rigidity theorem for weakly Landsberg manifolds}, Sci. China Ser. A, \textbf{50} (2007), 609--614.
	
	
	
\end{thebibliography}
\end{document}